\documentclass[a4,11pt]{article}
\usepackage{graphicx, bbding}
\usepackage{ mathrsfs, mathtools }  
\usepackage{psfrag}
\usepackage{subfigure}
\usepackage{upgreek}
\usepackage{amsmath}
\usepackage{amsfonts}
\RequirePackage[OT1]{fontenc}
\RequirePackage{amsthm,amsmath,amssymb,amstext,latexsym,epsfig}
\RequirePackage{txfonts, hypernat}

\numberwithin{equation}{section}
\theoremstyle{definition}
\newtheorem{The}{Theorem}[section]
\newtheorem{Pro}{Proposition}[section]
\newtheorem{Lem}{Lemma}[section]
\newtheorem{Cor}{Corollary}[section]
\theoremstyle{definition}
\newtheorem{Def}{Definition}[section]
\newtheorem{Rem}{Remark}[section]
\newtheorem{Ex}{Example}[section]

\newcommand{\X}{{\mathcal X}}
\newcommand{\ii}{{\mathit i}}
\newcommand{\R}{{\mathbb R}}

\newcommand{\N}{{\mathcal N}}
\newcommand{\wish}{{\mathcal W}}
\newcommand{\E}{{\mathbb E}}

\newcommand{\y}{{\mathbf y}}
\newcommand{\x}{{\mathbf x}}

\newcommand{\bo}[1]{\mathbf{#1}}
\newcommand{\PD}[1]{\mathrm{PD}_{#1}{(\mathbb{R})}}
\newcommand{\D}{\mathbb D}
\newcommand{\Complex}{\mathbb C}

\newcommand{\norm}[1]{\left\Vert#1\right\Vert}

\newcommand{\tr}{\mathrm tr}

\theoremstyle{definition}

\usepackage[tmargin=1.1in,bmargin=1.1in,rmargin=1.25in,lmargin=1.25in]{geometry}

\title{Maximal Invariants Over Symmetric Cones\protect}
\author{Emanuel Ben-David\\ {\small Stanford University} }
\date{August 2010}

\begin{document}

\maketitle

\begin{abstract}
In this  paper  we consider some hypothesis tests  within a family of Wishart distributions, where  both the sample space and  the parameter space are symmetric cones. For such testing problems, we first derive the joint density of the ordered eigenvalues of  the generalized Wishart distribution and propose a test statistic analog  to that of  classical multivariate statistics for testing  homoscedasticity of covariance matrix.   In this generalization of Bartlett's test for equality of variances to hypotheses of  real, complex, quaternion, Lorentz and octonion types of covariance structures.
\end{abstract}

\maketitle


\section{Introduction}
Consider a statistical model consisting of a sample space  $\X$  and unknown probability measures  $P_{\theta}$, with $\theta$ in the parameter space  $\Theta.$ In this setting, the statistical inference about the model is often concentrated on parameter estimation and  hypothesis testing.  In the latter, we typically test the hypothesis  $H_0: \theta\in\Theta_0\subset\Theta$  vs. the hypothesis $H: \theta\in \Theta\backslash \Theta_0$. Systematically, this means to find a suitable test statistic $t$ from $\X$ to a measurable space $Y,$ and the distribution of $tP_{\theta}$, the transformed measure  with respect to   $P_{\theta}$  under $t$.

For a {Gaussian model}, where the sample space is $\R^n,$ the probability measures are multivariate normal distribution $\N_n(0,\Sigma)$ and the parameter space is  $\PD{n}$, the cone of $n\times n$ positive definite matrices over $\R$, some classical examples of such hypotheses are the sphericity hypothesis:
\begin{equation}\label{bartest}
 H_0: \Sigma=\sigma^2 I_n\quad\text{vs.}\quad H: \Sigma\neq \sigma^2 I_n,
\end{equation} 
for some $\sigma>0$; the complex structure hypothesis 
\begin{equation}\label{comptest}
 H_0: \Sigma=\begin{pmatrix}
A&-B\\
B&A
\end{pmatrix}\in\mathrm{PD}_k(\Complex)\quad\text{vs.}\quad \Sigma\in\PD{2k},
\end{equation}
for some  real $k\times k$ matrices~$A,B$, with  $2k=n$; and the quaternion structure hypothesis 
\begin{equation}\label{quatertest}
H_0: \Sigma=\begin{pmatrix}
A& -B&- C&-D\\
B&A&-D &C\\
C&D&A&-B\\
D&-C&B&A
\end{pmatrix}\in\mathrm{PD}_m(\mathbb{H})\quad\text{vs.}\quad H: \Sigma\in\PD{4m},
\end{equation}
for some real $m\times m$ matrices $A,B,C,D$, with $4m=n$. 
Incidentally, in each of these hypotheses the parameter space in the full model is $\PD{n}$, the parameter space in the submodel is a subcone of $\PD{n},$ and the hypothesis is invariant under an action of a subgroup of the orthogonal linear group of $\R^n$, $\mathrm{O}(n,\R)$.\\

More generally, suppose $\Omega$ is a symmetric cone and
\begin{equation}\label{Model}
\mathcal{M}=\{P_{\sigma}\in\mathbb{P}(\X):~\sigma\in\Omega\}
\end{equation}
 is a statistical model with parameter space $\Omega$, and
 \begin{equation}\label{Submodel}
 \mathcal{M}_0=\{P_{\sigma}\in \mathbb{P}(\X):~\sigma\in \Omega_0\}
 \end{equation}
 is a submodel of $\mathcal{M}$, where the parameter set $\Omega_0$ is a  (symmetric) subcone of $\Omega$. In this set up, we wish to test the hypothesis
\begin{equation}\label{gtest}
H_0: \sigma\in\Omega_0\quad\text{vs.}\quad H:\sigma\in\Omega\backslash\Omega_0.
\end{equation}
In the presence of a group action, testing the hypothesis \eqref{gtest}  leads one to the study of a maximal invariant with respect to a group \cite{James1964,Muirhead1982,Andersson1983,Wijsman1985}.\\

 The maximal invariants that arise in testing problems like those stated in \eqref{bartest}, \eqref{comptest} and \eqref{quatertest} are indeed functions of the eigenvalues of the sample covariance matrix \cite{Andersson1983,Anderson1958,James1964,Muirhead1982}.
Therefore, in most situations, the problem is reduced to finding the distribution of the ordered eigenvalues of a distribution (often Wishart) over the cone of positive definite matrices over the real, complex, or quaternion fields.\\

In this respect, studies on the distribution of the ordered eigenvalues of a random matrix with values in a classical symmetric cone overlap with similar studies in {Random Matrix Theory (RMT)} \cite{Edelman1991, Edelman2005,Mehta2004}. Therefore, applications of such studies are not merely limited to the testing problems we mentioned. Other applications of the distribution of the ordered eigenvalues of a random matrix, specially a Wishart matrix, can be found in physics, e.g. \cite{Mehta2004}, principle component analysis, e.g. \cite{Muirhead1982,Johnstone2001}, and signal processing, e.g. \cite{Rao2007a, Rao2007, Ratnarajah2005}.\\

In this paper we discuss the statistical analysis of models of the form \eqref{Model}. There are many models that fall in this category, such as the Wishart, inverted Wishart, beta and hyperbolic distributions. In our analysis, we extensively exploit the algebraic and geometric structures of symmetric cones, described by  Jordan algebras and Riemannian symmetric spaces, respectively (see \cite{Jacobson1968,Faraut1994, Helgason1994}).    In the statistical inference of these models two problems of interest are the estimation of the parameter $\sigma$, and testing a hypothesis of the form \eqref{gtest}. Here we focus mainly on the latter. \\

 In approaching these problems, we begin in Section \ref{sec1} with giving a short review of symmetric cones and their analysis based on \cite {Faraut1994}. Essential parts of this section are the classification of irreducible symmetric cones, or equivalently, simple Jordan algebras, and Pierce decomposition.   In Section \ref{sec2} we study special functions over symmetric cones. Our main goal in this section is to define zonal polynomials over a symmetric cone. Sections \ref{sec3}, \ref{sec4}, and \ref{sec5} present our contribution to the theory of maximal invariants over symmetric  cones,  introducing a new family of non-central Wishart distributions over symmetric cones, deriving the joint distribution of the ordered eigenvalues of such distributions,  and proposing an analog of Bartlett's test for testing  homoscedasticity of, roughly speaking, the covariance matrix.
\section{The structure of symmetric cones}\label{sec1}
In this section we provide the reader with a condense review of  the structural analysis of symmetric cones which  is essential in understanding  the generalized notions, such as hypergeometric functions, zonal polynomials and the Wishart distribution,  we discuss later. More thorough presentation of these topics can be found in \cite{Faraut1994} and elsewhere in this volume.

\subsection{Symmetric cones}\label{sc}
Let $(J, \langle\cdot,\cdot \rangle)$ be a Euclidean vector space and  let $GL(J)$ be the set of regular linear transformations on $J$. A convex cone $\Omega\subset J$ is a set satisfying   $\alpha x+\beta y \in \Omega$  for  any  $\alpha, \beta>0$ and $x, y \in \Omega$.  Let  $\overline{\Omega}$ denote the closure of  $\Omega$  in $J$. The convex cone $\Omega$ is {proper} if $\overline{\Omega}\cap -\overline{\Omega}=\{0\}$,  full if $\Omega-\Omega=J$. A {homogeneous cone}  $\Omega$  is a full proper open convex cone with this property that the {automorphism group} of $\Omega$,  defined by  $G(\Omega)=\{g\in \mathrm{GL}(J):\: g\Omega=\Omega\}$, and consequently the connected component  of the identity in $G(\Omega)$, denoted by $G$,  act transitively on $\Omega$. The homogeneous cone  $\Omega$  is  called {symmetric} if it is homogeneous and self-dual, i.e.,  $\Omega^*=\{y\in V:\:\left\langle x|y\right\rangle >0, \forall x\in \overline{\Omega}\backslash\{0\}\}=\Omega$.  Every symmetric cone  is the finite product of irreducible symmetric cones, where a symmetric cone is irreducible if it is not the direct product of two or more symmetric cones. A familiar example of an irreducible symmetric cone is  $\mathrm{PD}_r(\R)$.\\

 A well-known property of a homogeneous cone  $\Omega$  is that the stabilizer of  each  $x\in\Omega$, $G_x$,  is a maximal compact subgroup of  $G$.  For a symmetric cone  $\Omega$, if we set $K=G\cap \mathrm{O}(J) $, where  $\mathrm{O}(J)$ is  the group of orthogonal transformation  of $J$,  then  $K$ is the stabilizer of an element of $\Omega$ and, consequently, a maximal subgroup of $G$.\\ 
 Another interesting property of a homogeneous cone $\Omega$  is that if we define $\varphi(x)=\int_{\Omega^*}\exp(-\left\langle x|y\right\rangle)dy$  for each  $x\in \Omega$, then $\varphi(x)dx$ is a $G$-invariant measure on $\Omega$. This follows from the fact that for any $g\in G$ we have $\varphi(gx)=|\text{det }g|^ {-1}\varphi(x).$
 \subsection{Euclidean Jordan algebras}\label{Ja}
  A  {Jordan algebra} $J$  over  $\R$   is a real vector space equipped with a multiplication satisfying  $xy=yx$ and  $x(x^2y)=x^2(xy)$  for any $x$,  $y$  in $J$. An {associative bilinear form}  $B$ on $J$ is a bilinear form satisfying 
\[
B(xy,z)=B(x,yz)\quad  \text{for any}  \  x,y,z \in J.
\] 
A  Jordan algebra with no non-trivial  ideal is called {simple}. Typical examples of Jordan algebras are so-called {special Jordan algebras}. If  $A$ is an associative algebra, then the special Jordan algebra $A^+$ is  the vector space $A$ equipped with the {Jordan product}  $x\circ y=(xy+yx)/2$. In particular $M_r(\R)$, the set of $r\times r$ matrices, and $\mathrm{S}_r(\R)$, the set of $r\times r$ symmetric matrices, are special Jordan algebras with the Jordan product.\\
 A  finite dimensional  Jordan algebra  $J$  over  $\R$  is called  {Euclidean}  if  there is  an associative inner product on  $J$.  Every Euclidean Jordan algebra has a multiplicative identity element $e$. The {quadratic representation} of $J$ is the map $P:J\rightarrow\mathrm{End}(J)$, where   $P: x\mapsto2L(x)^2-L(x^2)$   and $L(x)$  is the left multiplication map by $x$, i.e.,  $L(x)(y)=xy$. 
 
 \subsection{Spectral decomposition}\label{Sd}
  A non-zero element $c$ in  $J$ is {idempotent} if  $c^2=c$. An  idempotent is {primitive} if  it cannot be written as the sum of two idempotents. Two idempotents $c_1$ and $c_2$  are {orthogonal}  if $c_1c_2=0$. A {Jordan frame} is a maximal set of orthogonal primitive idempotents $c_1,\ldots, c_r$.  If  $J$ is simple, then for any Jordan frame  $c_1,\ldots, c_r$ we have  $c_1+\cdots+c_r=e$  and the number of elements in this Jordan frame, $r$, is invariant for every Jordan frame. This common number is called  the rank of $J$. Moreover, for each $x$ in $J$ there is a Jordan frame $c_1,\ldots,c_r$ such that  $\xi_1c_1+\cdots+\xi_rc_r=x$, where $\xi_1,\ldots,\xi_r$ are called the eigenvalues of  $x$. The trace and determinant of  $x$,  denoted by  $\tr(x)$  and $\det(x)$,  respectively, are the sum and product of all eigenvalues of $x$.
 \subsection{Euclidian Jordan algebras and symmetric cones} \label{EJ}
  For a Euclidean Jordan algebra   $J$  the set of squared elements $x^2$  where $x\in J$ is a closed cone. The interior of this cone, denoted by $J_+$,  is a symmetric cone  and is called the {cone of positive elements} of $J$. Equivalently,  $J_+$ is the set of  all $x^2$  such that  $x$  is regular, i.e.,  $\det(x)\neq 0$. Conversely, if $\Omega$ is  a symmetric cone and  $\epsilon\in \Omega$, then one can construct  a Euclidean Jordan algebra   $J$ such that $J_+=\Omega$ and $\epsilon$ is the identity element of $J$. 
 \subsection{Classification of irreducible symmetric cones} \label{CJ}
 The irreducible symmetric cones are in one-to-one correspondence with simple Euclidean algebras, which are classified into four families of classical Jordan algebras together with a single exceptional Jordan algebra.

  The first three families of classical Jordan algebras are matrix spaces. More specifically, let  $\D=\R$, $\Complex$ or the quaternion $\mathbb{H}$. Denote by  $\overline{x}$  the conjugate  of  $x$  in  $\D$,  $\Re{x}$  the real part of $x$, and  $\mathrm{H}_r(\D)$  the set of $r\times r$  Hermitian matrices  over $\D$. Recall that  $X^*$, the adjoint of a matrix  $X$, is obtained by taking the conjugate of each entries and then transposing the matrix. A matrix  $X$ is Hermitian if it is equal to $X^*$.  This space equipped with the Jordan product $X\circ Y=(XY+YX)/2$  and the scaler product  $\left\langle X|Y\right\rangle =\Re{\tr(XY)}$  is a simple Euclidean Jordan algebra of rank $r$ and corresponds to the irreducible symmetric cone of positive definite matrices over  $\R$,  $\Complex$  or $\mathbb{H}$, respectively.  The dimension of $\mathrm{H}_r(\D)$  over $\R$ is $ r+r(r-1)d/2$,  where   $d$, the  {Peirce constant}, is equal to the dimension of  the space $\D$ over $\R$.  When $\D$ is the set of real numbers $\mathrm{H}_r(\R)$  is indeed  $\mathrm{S}_r(\R)$  and   $\mathrm{S}_r(\R)_+$ is  $\PD{r}$. \\
  
 The fourth class of Jordan algebras is a Minkowski space. This is the vector space $\R\times \R^{n-1}$, $n>2$,  equipped with the product  $(\zeta,x)(\xi, y)=(\zeta\xi+x\cdot y,\zeta y+\xi x).$  It corresponds to the Lorentz cone $\mathbb{L}_n=\{(\zeta,x)\in \R\times\R^{n-1}:\: \zeta>\|x\|\}.$\\

The exceptional Jordan algebra can be described as follows. First notice that  $\mathbb{H}\times\mathbb{H}$  with the product  $(x_1,y_1)(x_2,y_2)=(x_1x_2-\overline{y_2}y_1,y_1\overline{x_2}+y_2x_1)$  is a non-associative algebra, called octonion  $\mathbb{O}$. Any element in $\mathbb{O}$ can be written as $x+jy$, where $j=(0,1)$ and $x,y\in \mathbb{H}$. The conjugation in $\mathbb{O}$ is defined by $\overline{x+jy}=\overline{x}-j\overline{y}$. The exceptional Jordan algebra is $\mathrm{H}_3(\mathbb{O})$  and its cone of positive  elements is denoted by $\mathrm{PD}_3(\mathbb{O})$.\\ The following table summarizes the information about simple Euclidean Jordan algebras \cite{Faraut1994,Ding1996}. 
\vspace{.7cm} 

\begin{table}[htdp]
\caption{ Classification of  simple Jordan algebras}
\begin{center}
\begin{tabular}{cccccccc}
\hline
$J$ & $J_+$ & $G$ & $K$ & \text{dim }J & \text{rank }J & $d$ \\
\hline
  $\mathrm{S}_r(\R)$ & $\PD{r}$ & $\mathrm{GL}^{+}_r(\R)$&$ \mathrm{SO}_r(\R)$&$\frac{1}{2}r(r+1)$ & $r$ & $1$ \\
  $\mathrm{H}_r(\Complex)$&$\mathrm{PD}_r(\Complex)$ & $\mathrm{GL}_r(\Complex)$& $\mathrm{U}_r(\Complex)$ &$ r^2$ &$ r$ &$ 2$ \\
  $\mathrm{H}_r(\mathbb{H})$ &$\mathrm{PD}_r(\mathbb{H})$ & $\mathrm{GL}_r(\mathbb{H})$&$\mathrm{Sp}_{2r}(\Complex)$ & $r(2r-1)$ & $r$ & $4$ \\
  $\R\times\R^{n-1}$& $\mathbb{L}_n $& $\R\times \mathrm{SO}^{+}_{1,n-1}(\R)$ & $\mathrm{SO}_{n-1},(R)$ & $n$ & 2 & $n-2$ \\
  $\mathrm{H}_3(\mathbb{O})$ & $\mathrm{P}_3(\mathbb{O})$ & $\R\times E_6$ & $F_4$ & $27$ & $3$ & $8$ \\
  \hline
  \end{tabular}
  \end{center}
  \label{class}
  \end{table}%
  \vspace{.6cm}
  \subsection{ Additional properties }\label{Jp}
{In the remainder of this section, and throughout the rest of the paper, we will assume that  $\Omega$ is an irreducible symmetric cone of positive elements of a simple Jordan algebra $J$ of dimension $n$, rank $r$  and the Peirce constant $d$  defined by  $n=r+r(r-1)d/2$}.\\ 
$a)$~For each $x\in \Omega$ and $g\in G$ we have
\begin{enumerate}
\item $\det  L(x)=\det (x)^{{n}/{r}}$, 
\item $\det  P(x)=(\det x)^{{2n}/{r}}$,
\item $\det gx=\det (g)^{{r}/{n}}\det(x)$ for each $g\in G$,
\item $\det P(y)(x)=(\det y)^2\det x$,
\item $(gx)^{-1}={g^*}^{-1}x^{-1}$,~\text{where} $g^*$ is the adjoint of $g$,
\item $ P(x)^{-1}=P(x^{-1})$,
\item $ P(x)^*=P(x)$, i.e., $P(x)$ is Hermitian.
\end{enumerate}
$b)$~$K$ acts transitively on the set of primitive idempotents and the set of all Jordan frames.
\section{Special functions on symmetric cones}\label{sec2}
The rich geometric structure of symmetric cones has naturally motivated many research studies in the field of harmonic analysis. In this section we give a short  description of  special functions on symmetric cones. 
\subsection{Gamma function}\label{Gf}
Let  fix a  Jordan frame   $c_1,\ldots,c_r$  in  $J$.  For each  $1\leq j \leq r$  we define the idempotent  $e_j= c_1+\cdots+c_j$. In general, it can be shown that if  $c$ is an idempotent element of  $J$, then the only possible eigenvalues  of the linear transformation $L(c)$  are  $0,1/2,1.$ One can see that the eigenspace,  $J(e_j,1)$, corresponding to the eigenvalue  $1$  of  $L(e_j)$ is a Jordan algebra with the multiplication inherited from  $J.$  Let   $\Omega_j$  be the cone of positive elements  and $\det^{(j)}$ the determinant  with respect to  this  Jordan algebra.  Let  $P_j: J\rightarrow J(e_j,1)$  be the orthogonal projection  on $J(e_j,1)$. The principal minor, $\Delta_j(x),$ is a homogeneous polynomial of degree $j$ on  $J$ defined by $\Delta_j(x) ={\det}^{(j)}(P_j(x))$. We extend this definition as follows. For each  $s=(s_1,\ldots,s_r)\in \mathbb{C}^r$ we set
\[
\Delta_s(x)=\Delta_1(x)^{s_1-s_2}\Delta_2(x)^{s_2-s_3}\cdots\Delta_r(x)^{s_r}.
\]
For each $s\in \Complex^r$  the gamma function is  $\Gamma_{\Omega}(s)=\int_{\Omega}\exp\{-\tr(x)\}\Delta_s(x)\det(x)^{-\frac{n}{r}}dx$. This integral is absolutely convergent if $\Re s_j>(j-1)d/2$, for $j=1,\ldots,r$. Moreover, 
\[
\Gamma_{\Omega}(s)=(2\pi)^{\frac{n-r}{2}}\prod_{j=1}^r{\Gamma (s_j-(j-1)\frac{d}{2})}.
\]
 In particular, identifying $z \in \mathbb{C}$ with $(z, \ldots,z)\in \mathbb{C}^r$, we have
 \[
 \Gamma_{\Omega}(z)=\int_{J_+}\exp\{-\tr(x)\}\det(x)^{z-\frac{n}{r}}dx
 \]
  is absolutely convergent if $\Re z>n/r-1$. If this is the case, then 
  \[
  \Gamma_{\Omega}(z)=(2\pi)^{\frac{n-r}{2}}\prod_{j=1}^r{\Gamma(z-(j-1)d/2)}.
  \]
  \subsection{Beta functions}\label{Bf}
  The{ beta function} on a symmetric cone $\Omega$ is defined by the integral
   \[
  B_{\Omega}(a,b)=\int\limits_{\Omega\cap (e-\Omega)}\negthickspace\Delta_{a-\frac{n}{r}}(x)\Delta_{b-\frac{n}{r}}(e-x)dx,
  \]  
  where $a,b \in \mathbb{C}^r$, and  $e-\Omega=\{e-x: x\in \Omega\}.$  This  integral converges absolutely if $\Re a_j> (j-1)d/2$  and  $\Re b_j> (j-1)d/2$. In this case  $B_{\Omega}(a,b)=\dfrac{\Gamma_{\Omega}(a)\Gamma_{\Omega}(b)}{\Gamma_{\Omega}(a+b)}$, and
\[
\int\limits_{\Omega\cap(x-\Omega)}\negthickspace\Delta_{a-\frac{n}{r}}(y)\Delta_{b-\frac{n}{r}}(x-y)dy=B_{\Omega}(a,b)\Delta_{a+b-\frac{n}{r}}(x).
\]
\subsection{The space of polynomials}\label{Pl}A partition is a finite sequence of non-negative integers $\lambda=(\lambda_1,\cdots, \lambda_r)$ in decreasing order  $\lambda_1\geq\cdots \geq \lambda_r.$  The {weight}  of  $\lambda$  is  $|{\lambda}|=\lambda_1+\cdots+ \lambda_r$. If $|\lambda|=k $, then we say  ${\lambda}$ is a partition of $k$.\\
 
 A  function  $f: J\rightarrow \R$ is  a {polynomial} on  $J$  if there is a basis  $\{v_1,\ldots,v_n\}$ of $J$  and a polynomial $p \in  \R[t_1,\ldots,t_n]$ such that for  any  linear combination  $x=\sum^n_{i=1}{\xi_i v_i} $,
\[
f(x)=p(\xi_1,\ldots,\xi_n).
\]
One can check that this definition is independent of the choice of basis. The set of all polynomials over $J$ is denoted by ${\mathcal P}(J)$. The action of  $G$  on  $J$ can be naturally extended to an action on ${\mathcal P}(J)$ by defining  $gp(x)= p(g^{-1}x)$.  Let  $\mathbb{P}_{\lambda}(J)$ be the subspace of  $\mathbb{P}(J)$ generated by polynomials  $g\Delta_{\lambda}$,   $g\in G.$ Then every $p$  in  $\mathbb{P}_{\lambda}(J)$ is  a homogeneous polynomial of degree  $|\lambda|$.
\subsection{Spherical polynomials}\label{Spl}
  Recall that  $K$,  the stabilizer  of the identity  $e$,   is a compact Lie subgroup of  $G$, thus there exists a Haar measure on  $K$. For each partition  $\lambda$  the spherical polynomial $\Phi_{\lambda}$  is  
\[ 
\Phi_{\lambda}(x)=\int_{K}\Delta_{\lambda}(kx)d\mu_K(k),
\]
where  $\mu_K$  is the normalized Haar measure on  $K$. The function  $\Phi_{\lambda}$  is indeed  a homogeneous polynomial of degree $|\lambda|$  and is invariant under the action of  $K$, i.e.,    $\Phi_{\lambda}(kx)=\Phi_{\lambda}(x)$  for  any  $k\in K$   and  $x\in J$. Moreover, the spherical polynomial $\Phi_{\lambda}$ is, up to a constant factor, the only $K$-invariant polynomial in $\mathcal {P}_{\lambda}(J)$. More precisely, if  $p$~is a  $K$-invariant homogeneous polynomial  in  $\mathbb{P}_{\lambda}(J)$, then
\begin{equation}\label{sph1}
\int_K p(kx)d\mu_K(k)=p(e)\Phi_{\lambda}(x).
\end{equation}
Consequently, for any $g\in G$ 
\begin{equation}\label{sph2}
\int_K \Phi_{\lambda}(gkx)d\mu_K(k)=\Phi_{\lambda}(ge)\Phi_{\lambda}(x).
\end{equation}

If  $x\in \Omega$ and $\Re{\gamma}>(r-1)d/2$, then for any $y$ in $\Omega$ and  $g \in G$  we have
\begin{equation}\label{sp3}
\int_{\Omega} e^{-\tr(xy)}\Phi_{\lambda}(gx)\det(x)^{\gamma-\frac{n}{r}}dx=\Gamma_{\Omega}(\lambda+\gamma)\Delta^{-\gamma}(y)\Phi_{\lambda}(gy^{-1}).
\end{equation}
\subsection{Zonal polynomials}\label{Zpl}
The {Pochhammer symbol} for  $\Omega$ is defined by   
\[
(s)_\lambda=\dfrac{\Gamma_{\Omega}(s+\lambda)}{\Gamma_{\Omega}(s)},\quad s\in \Complex.
\]
This definition generalizes the classical  Pochhammer symbol. 
 The {zonal polynomial}  $Z_{\lambda}$ is  a homogeneous $K$-invariant polynomial  on $J$ defined by
\begin{equation}
Z_{\lambda}(x)= d_{\lambda}\dfrac{|\lambda|!}{(\frac{n}{r})_{\lambda}}\Phi_{\lambda}(x),
\end{equation}
where $d_{\lambda}$ is the dimension of the vector space ${\mathcal P}_{\lambda}(J)$, and $\Phi_{\lambda}$ is the spherical polynomial. Zonal polynomials are $K$-invariant  polynomials normalized by the property
\begin{equation}
\tr(x)^k=\sum_{|\lambda|=k}Z_{\lambda}(x)\quad x\in \Omega.
\end{equation}
Notice that the function  $p(x)=\tr(x)^k$ on $\Omega$ is a $K$-invariant homogeneous polynomial in ${\mathcal P}
_{\lambda}(J)$.  It follows from Equation \eqref{sph2} that for each  $x\in J$  and  $y\in \Omega$
\begin{equation}\label{z4}
\int_{K} Z_{\lambda}(P(y^{\frac{1}{2}})kx)d\mu_K(k)=\dfrac{Z_{\lambda}(y)Z_{\lambda}(x)}{Z_{\lambda}(e)}.
\end{equation}
\subsection{Hypergeometric functions }\label{Hpl}
Let $a_1,\ldots, a_p$ \  and \  $ b_1,\ldots, b_q$ be real numbers with  $a_i-(i-1)d/2\geq 0, b_j-(j-1)d/2\geq 0$  and  $x,y\in J$. The {hypergeometric function} $_pF_q$  is defined by
\[
_pF_q(a_1,\ldots,a_p, b_1,\ldots, b_q, x,y)=
\sum_{k=1}^{\infty}\sum_{|\lambda|=k}\dfrac{(a_1)_{\lambda}\cdots (a_p)_{\lambda}}{(b_1)_{\lambda}\cdots (b_q)_{\lambda}}\dfrac{Z_{\lambda}(x)}{k!}\dfrac{Z_{\lambda}(y)}{Z_{\lambda}(e)}.
\]
For $y=e$  we set  ~$ _pF_q(a_1,\ldots,a_p, b_1,\ldots, b_q, x)={_pF_q}(a_1,\ldots,a_p, b_1,\ldots, b_q, x,e). $ This means that
\begin{eqnarray*}
_pF_q(a_1,\ldots,a_p, b_1,\ldots, b_q, x)=
\sum_{k=1}^{\infty}\sum_{|\lambda|=k}\dfrac{(a_1)_{\lambda}\cdots (a_p)_{\lambda}}{(b_1)_{\lambda}\cdots (b_q)_{\lambda}}\dfrac{Z_{\lambda}(x)}{k!}.
\end{eqnarray*}
  This series converges absolutely if $p\leq q$ and diverges if $p>q$. Furthermore,  $_0F_0(x)=e^{\tr(x)}$  and  $_0F_1(b, x)=\det(e-x)^{-b}$.\\
Suppose  $x\in \Omega$ and  $g \in G$.  By  using  the integral in  Equation \eqref{z4}  for  $p\leq q$ we get
\begin{equation}\label{eq:kp}
\int_K { _pF_q(a_1,\ldots,a_p,b_1,\ldots,b_q, g(kx))d\mu_K(k)}= {_pF_q(a_1,\ldots,a_p,b_1,\ldots,b_q, x,ge)}.
\end{equation}
Similarly, if  $\Re \gamma> (r-1)d/2$, $y \in \Omega$   and   $p<q$, or   $y\in e-\Omega$   and   $p=q$, then by applying Equation \eqref{sp3}  we get
\begin{eqnarray*}
&&\int_{\Omega} e^{-\tr(xy)} {_pF_q(a_1,\ldots,a_p,b_1,\ldots,b_q,  gx, z)}\det(x)^{\gamma-\frac{n}{r}}dx=\\
&&\Gamma_{\Omega}(\gamma)\det(y)^{-\gamma} {_{p+1}F_q(a_1,\ldots,a_p, \gamma, b_1,\ldots,b_q, gy^{-1}, z)}.
\end{eqnarray*}
We record the following Proposition, later used in deriving the beta type distribution, from \cite{Faraut1994}.
\begin{Pro}\label{hyperg}
Suppose that $ p\leq q+1$,  $\Re \eta_1>(r-1)d/2$ and  $\Re \eta_2>(r-q)d/2$.  If $g \in G$  such that $ge\in \Omega\cap (e-\Omega)$ , then
\begin{eqnarray}\label{b1} 
&&\int_{\Omega\cap ( e-\Omega)} {_pF_q(a_1,\ldots,a_p, b_1,\ldots, b_q, gx)}\det(x)^{\eta_1-\frac{n}{r}}\det(e-x)^{\eta_2-\frac{n}{r}}dx\\
\notag&=&B_{\Omega}(\eta_1, \eta_2){_{p+1}F_{q+1}(a_1,\ldots,a_p, \eta_1, b_1,\ldots, b_q, \eta_1+\eta_2, ge)}.
\end{eqnarray}
\end{Pro}
 Suppose  $y\in J$  and  $x\in\Omega$. We define  $x\star y=P(x^{\frac{1}{2}})y$ and
  $\mathrm{Eig(y|x)}$  equal to the $r$-tuple of the ordered eigenvalues of $y$  with respect to  $x$, in decreasing order. Notice that these eigenvalues are the roots of the polynomial
$
p(\ell)=\det(\ell x-y).
$
In particular, for $x=e$, we have  $\mathrm{Eig}(y)=\mathrm{Eig}(y|e)$  is the ordered eigenvalues of $y$. 
Note that for a special Jordan algebra,  such  as the space of symmetric matrices or Hermitian ones, the  operation  $x\star y$  is   $x^{\frac{1}{2}}yx^{\frac{1}{2}}$. Here we list some interesting properties of  $\star$ operation.
\begin{Lem}
Let $x$ and $y$ be in $\Omega$. Then\\
$a)$~ $x\star e=e\star x=x$,  $x\star x^{-1}=x^{-1}\star x=e$ and  $(x\star y)^{-1}=x^{-1}\star y^{-1}$.\\
$b)$~$x\star y$ and $y\star x$ have the same eigenvalues, therefore,  any $k$-invariant function on $\Omega$ returns the same value on both $x\star y$ and $y\star x$. In particular,
$\det(x\star y)=\det(y\star x)=\det(x)\det(y),~\tr(x\star y)=\tr(y\star x)=\tr(xy)~\text{and}~Z_{\lambda}(x\star y)=Z_{\lambda}(y\star x).$
\end{Lem}
\begin{proof}
$a)$~ We only prove the last equality. The other equalities are immediate from the definition of  $x\star y$.   From  property   $(5)$  in  Subsection \ref{Jp} we get
\[
(x\star y)^{-1}=({P(x^{\frac{1}{2})}y})^{-1}={P(x^
{\frac{1}{2}})^*}^{-1}y^{-1}
=P({x^{-\frac{1}{2}}})y^{-1}
=x^{-1}\star y^{-1}.
\]
$b)$~ For an  $r$-tuple  $(\xi_1,\ldots,\xi_r)$  with $\xi_1>\cdots\> \xi_r>0$,  we define 
\[
(\xi_1,\ldots,\xi_r)^{-1}=(\xi^{-1}_r,\ldots,\/\xi^{-1}_1).
\]  
By this notation, one can check that for any $x\in \Omega$ we have  $\mathrm {Eig}(x^{-1})=\mathrm {Eig}(x)^{-1}$. Clearly, the ordered eigenvalues of $x\star y$   are  equal to $\mathrm {Eig}(y|x^{-1})$. On the other hand, $\mathrm {Eig}(y|x)$ is equal to  $\mathrm {Eig}(x|y)^{-1}$ (if  $\ell$ is an eigenvalue of $x$  with respect to  $y$, then  $1/\ell$  is an eigenvalue of $y$  with respect to  $x$). Therefore, $
  \mathrm {Eig}(x\star y)=\mathrm {Eig}(y|x^{-1})$
  $=\mathrm {Eig}(x^{-1}|y)^{-1}
 =\mathrm {Eig}(y^{-1}\star x^{-1})^{-1}
 =\mathrm {Eig}((y^{-1}\star x^{-1})^{-1})
  =\mathrm {Eig}(y\star x).$
\end{proof}
\begin{Rem}
There is worth mentioning  an interesting fact about the  $x\star y$  operation  on  $\Omega$.  Suppose  $\mu$  and  $\nu$  are two   $K$  invariant measures on  $\Omega$.   These can be lift up, respectively,  to   $K$  bi-invariant measures  $\mu^\sharp$  and   $\nu^\sharp$  on   $G$.  By projecting the convolution measure  $\mu^\sharp\star\nu^\sharp$  on  $\Omega$   we obtain a measure  on   $\Omega$ , denoted by  $\mu\star\nu$   and called the convolution of   $\mu$  and  $\nu$.  Now  if  $\x$  and   $\y$  are two random vectors in  $\Omega$  with  $K$  invariant distributions   $\mu$  and   $\nu$, then   the distribution  of the random  vector  $\x\star\y$ is  the convolution   $\mu\star\nu$  \cite{Graczyk2001}.
   \end{Rem}
\section{The non-central Wishart distribution }\label{sec3}
In multivariate statistics, the  so-called  {non-central Wishart distribution} is the multivariate analog of the {non-central chi-square distribution} and, similarly, it arises out of the sampling distribution of a sample statistic, namely  the sample covariance matrix of a multivariate normally distributed population. More generally, if   $X$ is an~$N\times r$  random matrix  such that $X\sim\N_{N\times r}(M, I_N\otimes\Sigma)$  with  $\Sigma\in \mathrm{PD}_r(\R)$, then the density of $S= X^TX$   with respect to  the standard Lebesgue measure  is  proportional to
\[
\det({\Sigma})^{-\frac{N}{2}}\exp\{-\frac{1}{2}\tr (\Sigma^{-1}S)\}\exp\{-\frac{1}{2}\tr(\Delta)\} _0F_1 (\frac{1}{2}N, \frac{1}{4}\Delta\Sigma^{-1}S)\mathbf{1}_{\PD{r}}(S).
\]
 Thus the non-central Wishart distribution is the transformed distribution  with respect to  a  multivariate normal distribution under the quadratic mapping  $X\mapsto X^TX$. We can generalize this situation as follows.

 Let $J$ be a simple Euclidean  Jordan algebra, and $(E, \langle \cdot|\cdot\rangle_E)$ a  Euclidean space. A {(Jordan) representation} $\psi$ of $J$ over $E$ is a linear mapping $\psi:J\rightarrow \mathrm{End}(E)$  such that  $\psi(x^2)=\psi(x)^2$ for each  $x\in J$. The  representation $\psi$ is called self-adjoint if the linear transformation  $\psi(x)$ is self-adjoint for every $x\in J$.\\
Assume $\psi$ is a self-adjoint representation of $J$ on the Euclidean vector space $E$. For each $v\in E$, the function  $x\mapsto \langle \psi(x)v|v\rangle_E$ is a linear form on $E$. Therefore, there exists a quadratic map  $Q: E\rightarrow J$  such that $\langle \psi(x)v|v\rangle_E=\langle x|Q(v)\rangle_J.$ The fact that $\langle Q(v)|x^2\rangle_J=\norm {\psi(x)v}_E^2> 0$ implies that $Q(v)\in \overline{\Omega}$, for each $v\in E$.
\begin{Pro}\label{ncwPro}
Let $\bo{v}\sim \N(\mu, \psi(\sigma))$ be a normally distributed random vector in  $E$  with mean vector $\mu \in E$ and covariance $\psi(\sigma)\in \mathrm{PD}(E)$, where  $\psi$ is a self-adjoint representation of  the simple Jordan algebra $J$  over $E$,  $\sigma\in \Omega$,  and $N>2(n-r)$.  Then the density of  $\x=Q(\bo{v})$   with respect to  the canonical Lebesgue measure on $\Omega$ is 
\begin{equation}\label{ncw}
p_{\x}(x)=\left(2^{\frac{N}{2}}\Gamma_{\Omega}(\frac{N}{2r})\det(\sigma)^{\frac{N}{2r}}\right)^{-1}\exp\{-\frac{1}{2}\tr(\epsilon)\}
\cdot\det(x)^{\frac{N}{2r}-\frac{n}{r}}{_0F_1}(\frac{N}{2r}, \frac{1}{4}( \sigma^{-1}\star \epsilon)\star x)
\end{equation}
where $\epsilon=\sigma^{-1}\star Q(\mu)$.
\end{Pro}
\begin{proof}
 First we compute  $g(z)=\E[\exp\{ -\langle z, Q(v)\rangle_J\}]$, the Laplace transform of $Q(\bo{v})$.  For each  $z\in J$  we have  $g(z)$  is equal to
\begin{align*}
&\left((2\pi)^{\frac{N}{2}}\det(\psi(\sigma))^{\frac{1}{2}}\right)^{-1}\int_V \exp\{ -\langle \psi(z)v,  v\rangle_E\}\exp\{ -\frac{1}{2}\langle \psi(\sigma)^{-1}(v-\mu), (v-\mu)\rangle_E \}dv\\
 &=\left((2\pi)^{\frac{N}{2}}\det(\sigma)^{\frac{N}{2r}}\right)^{-1}
\cdot\int_V \exp\{-\frac{1}{2}\left( 2\langle \psi(z)v, v\rangle_E+\langle \psi(\sigma)^{-1}(v-\mu), (v-\mu)\rangle_E\right)\}dv.
\end{align*}
We rewrite the expression
$2\langle \psi(z)v, v\rangle_E+\langle \psi(\sigma)^{-1}(v-\mu), (v-\mu)\rangle_E$
as
\begin{align*}
&\langle \psi(2z+\sigma^{-1})(v-\psi(2z+\sigma^{-1})^{-1}\psi(\sigma^{-1})\mu, 
 (x-\psi(2z+\sigma^{-1})^{-1}\psi(\sigma^{-1})\mu\rangle_E\\
 & +\langle \psi(2z+\sigma^{-1})^{-1}\psi(\sigma^{-1})\mu, \psi(\sigma^{-1})\mu\rangle_E
 +\langle\psi(\sigma^{-1})\mu,\mu\rangle_E,
 \end{align*}
and substitute it in the original equations, obtaining
$$
g(z)=2^{-\frac{N}{2}}\det(\sigma)^{-\frac{N}{2r}}\det(z+\frac{1}{2}\sigma^{-1})^{-\frac{N}{2r}}\exp\{-\frac{1}{2}\tr(\epsilon)\}
 {_0F_0}(\frac{1}{4}(\sigma^{-1}\star\epsilon)\star(z+\frac{1}{2}\sigma^{-1})^{-1}),
$$
where $\Re (z+\frac{1}{2}\sigma^{-1})>0$.  The density of $\x=Q(\bo{v})$ in Equation \eqref{ncw} then follows from the inverse Laplace transform of $g(z)$ (see \cite{bendavid2008}).
\end{proof}

In light of Proposition \ref{ncwPro},   Equation  \eqref{ncw} can be taken as the density of the non-central Wishart distribution over the irreducible symmetric cone $\Omega$. 
\begin{Def} A random vector $\x$ with values in an irreducible symmetric cone $\Omega$ is said to have the Wishart distribution with parameters $\eta>n-r$, $\sigma \in \Omega$ and $\epsilon  \in \overline{\Omega}$, if its probability density   with respect to  the standard canonical Lebesgue measure on  $J$    is given by
\begin{eqnarray*}
\exp\{-\frac{1}{2}\tr(\epsilon)\}{_0F_1}(\frac{1}{2}\eta, \frac{1}{4}( \sigma^{-1}\star \epsilon)\star x)w_{\Omega}(x|\eta,\sigma)\bo{1}_{\Omega}(x),
\end{eqnarray*}
where  
\begin{equation}
w_{\Omega}(x|\eta,\sigma)=\dfrac{1}{2^{\frac{1}{2}\eta r}\Gamma_{\Omega}(\frac{\eta}{2})\det(\sigma)^{\frac{1}{2}\eta}}\exp\{-\frac{1}{2}tr(\sigma^{-1}x)\}
\det(x)^{\frac{1}{2}\eta-\frac{n}{r}}1_{\Omega}(x). 
 \end{equation}
This is denoted by $\x\sim \wish_{\Omega}(\eta, \sigma, \epsilon)$, where $\eta$, $\sigma$, and $\epsilon$ are,  respectively, the  {shape}, (multivariate) scale, and  {non-centrality parameters}. In particular, if $\epsilon=0$, the distribution is called the {(central) Wishart distribution} over a symmetric cone \cite{Casalis1996}, denoted by
\[
\x\sim \wish_{\Omega}(\eta, \sigma),
\]
otherwise, it is called the {non-central Wishart distribution} over a symmetric cone.
\end{Def}
\begin{Rem} If $\x\sim \wish_{\Omega}(\eta, \sigma)$, then  its probability density is  $w_{\Omega}(x|\eta,\sigma)$.
\end{Rem}
\begin{Rem}
In accordance with classification of irreducible symmetric cones, the non-central Wishart distribution over an irreducible symmetric cone can be categorized  as  :  $i)$~the real type  $\wish_r(\eta, \sigma,\epsilon)$   over  $\PD{r}$,  $ii)$~ the complex type $\Complex\wish_{r}(\eta,\sigma,\epsilon)$  over  $\mathrm{PD}_r(\Complex)$,  $iii)$~ the quaternion type  $\mathbb{H}\wish_r(\eta,\sigma,\epsilon)$  over  $\mathrm{PD}_r( \mathbb{H})$, $iv)$~the Lorentz  type  $\mathbb{L}\wish_n(\eta, \sigma,\epsilon)$  over  the Lorentz cone  $\mathbb{L}_n$, and  $v)$  the exceptional  type  $\mathbb{O}\wish(\eta, \sigma,\epsilon)$ over  the octonion cone $\mathrm{PD}_3( \mathbb{O})$. 
\end{Rem}
To explore some of the properties of the non-central Wishart distribution, resembling those of the classical non-central Wishart, first  we compute the characteristic function of the non-central Wishart distribution. By definition this is equal to  $\phi_x(z)=\E[\exp\{\ii \tr(xz)$. Thus we have
\begin{align*}
\phi_{\x}(z)&=\left(2^{\frac{1}{2}\eta r}\Gamma_{\Omega}(\frac{\eta}{2})\det(\sigma)^{\frac{1}{2}\eta}\right)^{-1}\exp\{-\frac{1}{2}\tr(\epsilon)\}\int_{\Omega}\exp\{-\frac{1}{2}\tr((\sigma^{-1}-2\ii z)x)\}\\
&\cdot{_0F_1}(\frac{1}{2}\eta, \frac{1}{4}( \sigma^{-1}\star \epsilon)\star x)\det(x)^{\frac{1}{2}\eta-\frac{n}{r}}dx\\
&=\left(2^{\frac{1}{2}\eta r}\det(\sigma)^{\frac{1}{2}\eta}\right)^{-1}\exp\{-\frac{1}{2}\tr(\epsilon)\}
\det(\dfrac{\sigma^{-1}-2\ii z}{2})^{-\frac{1}{2}\eta}\\
&\cdot{_1F_1}(\frac{1}{2}\eta),\frac{1}{2}\eta, \frac{1}{4}(\dfrac{\sigma^{-1}-2\ii z}{2})^{-1}\star(\sigma^{-1}\star\epsilon)).
\end{align*}
Simplifying this we get the characteristic function of $\x$ is given by
\[
\phi_x(z)=\det(e-2\ii \sigma z)^{\frac{1}{2}\eta}\exp\{ -\frac{1}{2}\tr(\epsilon)\}\exp\{\frac{1}{2}tr(e-2\ii \sigma z)^{-1}\epsilon)\}.
\]
Using this we can show that the convolution of two independently distributed Wishart is a Wishart distribution again. More precisely we have:
\begin{Cor}\label{Cor1}
Suppose $\x\sim \wish_{\Omega}(\eta_1, \sigma, \epsilon_1)$ and $\y\sim \wish_{\Omega}(\eta_2, \sigma, \epsilon_2)$ are independently distributed. Then $\x+\y\sim \wish_{\Omega}(\eta_1+\eta_2, \sigma, \epsilon_1+\epsilon_2)$.
\end{Cor}
 \begin{proof}
 The characteristic function of $\x+\y$ is the product of the characteristic functions of $\x$ and $\y$. It remains to check that this product is equal to the characteristic function of $\wish(\eta_1+\eta_2,\sigma,\epsilon_1+\epsilon_2)$.
 \end{proof}
Other properties of the non-central Wishart distribution on an arbitrary irreducible symmetric cone are similar to those of the classical non-central Wishart distribution. Here we list some of these properties that can be shown to hold as in Proposition 3.1. of \cite{Letac2007}. \\
$a)$~ For every  $g\in G$  we have $g\x\sim \wish_{\Omega}(\eta, g\sigma, {g^{*}}^{-1} \epsilon)$.\\
$b)$~ The Laplace transform of $\wish_{\Omega}(\eta, \sigma, \epsilon)$ is 
    \[
    \E[e^{-\tr(zx)}]=\det(e+2 \sigma z)^{\frac{1}{2}\eta}\exp\{ -\frac{1}{2}tr(\epsilon)\}\exp\{\frac{1}{2}tr(e+2\sigma\star z)^{-1}\epsilon)\}.
    \]
  $c)$~The {natural exponential family}  generated by $\wish_{\Omega}(\eta,e, \epsilon)$ is \ $\{ \wish_{\Omega}(\eta, \sigma, \epsilon): \sigma\in \Omega\}$. In particular, the statistical model
\[
\{\wish_{\Omega}(\mu,\sigma,\epsilon)\in \mathbb{P}(\Omega): (\mu,\sigma,\epsilon)\in J\times\Omega\times \overline{\Omega}\}
\]
  is an exponential family.\\
 $d)$~ $\E[x]=\eta\sigma+ \sigma\star \epsilon$.\\
For the remainder of  this section we state and prove two propositions that we will need later in  Section \ref{sec4}. The following proposition is proved in \cite{bendavid2008}.
  \begin{Pro}\label{p1}
 Suppose $\x\sim \wish_{\Omega}(\eta_1,e,\epsilon_1)$ and $\y\sim \wish_{\Omega}(\eta_2, e, \epsilon_2)$  are stochastically independent. Then the density of
 $\bo{z}=\y^{-1}\star \x$ is given by
 \begin{eqnarray*}
 &&\frac{1}{B_{\Omega}(\frac{1}{2}\eta_1,\eta_1+\frac{1}{2}\eta_2)}\det(z)^{\frac{1}{2}\eta_1-\frac{n}{r}}\det(e+z)^{-(\frac{1}{2}\eta_1+\frac{1}{2}\eta_2)}\exp\{-\frac{1}{2}
\tr(\epsilon_1+\epsilon_2)  \}\\
&&\cdot{_{1}F_1(\frac{1}{2}\eta_1+\frac{1}{2}\eta_2, \frac{1}{2}\eta_2,\frac{1}{2}\epsilon_2\star (e+z)^{-1}){_0F_1}(\frac{1}{2}\eta_1, \frac{1}{4}\epsilon_1\star z)\bo{1}_{\Omega}(z)}.
\end{eqnarray*}
 \end{Pro}
 \begin{Pro}
 Suppose $\x\sim \wish_{\Omega}(\eta_1,e,\epsilon_1)$ and $\y\sim \wish_{\Omega}(\eta_2, e, \epsilon_2)$ are independently distributed. Then the density of
 $\bo{u}=(\x+\y)^{-1}\star \x$ is given by
 \begin{eqnarray*}
 &&\frac{1}{B_{\Omega}(\frac{1}{2}\eta_1,\eta_1+\frac{1}{2}\eta_2)}\det(u)^{\frac{1}{2}\eta_1-\frac{n}{r}}\det(e+u)^{\eta_1+\frac{1}{2}\eta_2-\frac{n}{r}}
\exp\{-\frac{1}{2}\tr(\epsilon_1+\epsilon_2)  \}\\
&&\cdot{_{1}F_1(\eta_1+\frac{1}{2}\eta_2, \frac{\eta_1+\eta_2}{2},\frac{1}{2}\epsilon_2\star (e+u)^{-1}){_0F_1}(\frac{1}{2}\eta_1, \frac{1}{4}\epsilon_1\star z)\bo{1}_{\Omega}(u)}.
\end{eqnarray*}
 \end{Pro}
\begin{proof}
        If we set  $\bo{s}=(\x+\y)$ and $\bo{w}=\bo{u}|\bo{s}$, then by Corollary \ref{Cor1} and the property $a$)  above, we get
        \[
        \bo{s}\sim \wish_{\Omega}(\eta_1+\eta_2,e, \epsilon )\quad\text{and} \quad \bo{w}\sim \wish_{\Omega}(\eta_1, s^{-1}).
        \]
The rest of the proof is similar to the proof of Proposition \ref{p1}
\end{proof}

\section{Maximal invariants over symmetric cones}\label{sec4}
 
Our primary goal in this section is to introduce a maximal invariant statistic for a family of Wishart distributions parameterized over an irreducible symmetric cone $\Omega$. In reference to our discussion on Jordan algebras we can think of $\Omega$ as  being the open cone of positive elements of a simple Euclidean Jordan algebra $J$. This identification, necessarily, depends on our choice of an element $e\in \Omega$ which becomes the identity element of $J$.\\

To restate our assumptions, let $J$ be a simple Euclidean Jordan algebra of dimension $n$, rank $r$ and the Pierce constant $d$ given
by the formula $n=r+r(r-1)d/2$. Let  $e$  be the identity element and  $\Omega$  the corresponding open cone of positive elements of $J$, i.e., $\Omega=\{x^2| x\in J,
\det(x)\neq 0\}$.
 As before, $K$  is the stabilizer of  $e$   in  $G$  and  $c_1,c_2,\ldots,c_r$ is a Jordan frame in $J$. Recall that the group $K$ acts transitively on the set of Jordan frames, therefore, for each $x \in J$  there exists a $k\in K$ such that
\begin{equation}\label{polarDec}
x=k(\xi_1c_1+\cdots+\xi_rc_r),
\end{equation}
where $\xi_1,\ldots,\xi_r$ are the eigenvalues of $x$. The decomposition of $x$ in  Equation  \eqref{polarDec}  is called the {polar decomposition} of $x$  with respect to  the Jordan frame $c_1,\ldots, c_r$.  The polar decomposition is the main tool in deriving the joint density of the ordered eigenvalues of a random vector in $\Omega.$ Here we need  the following Theorem from \cite{Faraut1994}.
\begin{The}\label{polar}
Let $f$ be an integrable function on  the  Jordan algebra $J$. Then
\begin{equation}\label{polarInt}
\int_{J}f(x)dx=c_o\int_{\R^r_{\geq}}(\int_Kf(k\sum^r_{j=1}\xi_jc_j)\prod_{j<i}(\xi_j-\xi_i)^dd\mu_K(k))d\xi_1\cdots d\xi_r,
\end{equation}
 where $\mu_K$ is the normalized Haar measure on the compact Lie group $K$, $c_0$ a constant  depending only on  the  Jordan algebra $J$  and~
$ \R^r_{\geq}=\{(\xi_1,\ldots ,\xi_r)\in \R^r: \xi_1\geq\cdots\geq\xi_r\}.$
\end{The}
Notice that if $\x$ is a random vector with values in  $\Omega$  and   $var(\x)\neq 0$, then the ordered eigenvalues of $\x$, almost surely, is a random vector with values in
\[
\R^r_+=\{(\xi_1,\ldots,\xi_r)\in \R^r: \xi_1>\cdots>\xi_r> 0\}.
\]
Using the integral form \eqref{polarInt} we can write a general formula for calculating the probability density of the ordered eigenvalues of $\x$.  Consider the proper action of $K$ on the symmetric cone $\Omega$ and the mapping $\mathrm{Eig}:\Omega \rightarrow \R^r_+~(x\longmapsto \mathrm{Eig}(x)=(\xi_1(x),\ldots\xi_r(x)).$
  Clearly, $\mathrm{Eig}(x)$ is a maximal invariant map under the action of $K$ on $\Omega$. The distribution of $\mathrm{Eig}(\x)$ is the joint distribution of the ordered eigenvalues of $\x$.
Let  $\nu$ denote  the $G$-invariant  measure on $\Omega$ with the density $\det(x)^{-\frac{n}{r}}\bo{1}_{\Omega}(x).$
 By replacing the integrable function   $f$   in  Theorem \ref{polar} with  $f(x){\det}(x)^{-\frac{n}{r}}\bo{1}_{\Omega}(x)$  and applying  the integral formula \eqref{polarInt} one can see that the density of the quotient measure $\nu/\mu_K$,  with respect to  the Lebesgue measure $d\xi_1\cdots d\xi_r$  on $\R^r_+$,  is 
\[
c_0\prod_{j=1}^r\xi_j^{-\frac{n}{r}}\prod_{j<i}^r(\xi_j-\xi_i)^d.
\]
 On the other hand, by using Equation $(7)$  in \cite{Andersson1982},  we get the density of $\mathrm {Eig}(\x)$  with respect to  the quotient measure $\nu/\mu_K$ is given by
$\prod_{j=1}^r\xi_j^{\frac{n}{r}}\int_Kp_{\x}(k\sum^r_{j=1}\xi_jc_j)d\mu_K(k),$
where $p_\x$ is the density of  the random vector $\x$.
Combining these two, we obtain the density of $\mathrm{Eig}(\x)$ as
\[
c_0\prod_{j<i}(\xi_j-\xi_i)^d\int_{K}p_{\x}(k\sum_{j=1}^r\xi_j c_j)d\mu_K(k).
\]
It remains to compute the constant $c_0$, which, from Theorem $\ref{polar}$, does not depend on the choice of $f$. Therefore, using  the Selberg integral \cite{Koranyi1982}, \small
\[
\int_0^{\infty}\ldots\int_0^{\infty}\exp\{-\sum_{j=1}^r a_j\}\prod_{j=1}^r a_j^{x-1}\prod_{j<i}|a_j-a_i|^{2z}da_1\ldots d_r=\prod_{j=1}^r \dfrac{\Gamma(x+(j-1)z)\Gamma(jz+1)}{\Gamma(z+1)},
\]
\normalsize 
we obtain  $c_0=\dfrac{(2\pi)^{n-r}\Gamma(\frac{d}{2})^r}{\Gamma_{\Omega}(\frac{rd}{2})}$.  In summary,
 \begin{The}\label{deigen}
 Suppose $\x$ is a random vector with values in $\Omega$  and  $p_{\x}$ the density of  $\x.$ Then the joint density of the ordered eigenvalues of
$\x$ , i.e.,  \ $\xi_1(\x)> \cdots> \xi_r(\x)>0$ is given by
\[
\dfrac{(2\pi)^{n-r}\Gamma_{\Omega}(\frac{d}{2})^r}{\Gamma_{\Omega}(\frac{rd}{2})}\;\prod_{j<i}(\xi_j-\xi_i)^d\int_{K}p_{\x}(k\sum_{j=1}^r\xi_j c_j)d\mu_K(k).
\]
\end{The}
The following Corollary is proved in  \cite{bendavid2008}.
\begin{Cor}
Let $\x$  be a random vector with the Wishart distribution $\wish_{\Omega}(\eta, \sigma, \epsilon)$.\\
$a)$~ If $\epsilon=0$, i.e., the distribution is central, then the density of  the  maximal invariant $\mathrm{Eig}(\x)$, thus the joint density of the distribution of the ordered eigenvalues of $\x$,  is given by
 \[
 c_0\dfrac{\prod_{j<i}(\xi_j-\xi_i)^d\prod^r_{j=1}{\xi_j}^{\frac{1}{2}\eta-\frac{n}{r}}}{2^{\frac{1}{2}\eta r}\Gamma_{\Omega}(\frac{\eta}{2})\det(\sigma)^{\frac{1}{2}\eta}}
{_0F_0}(x, \frac{1}{2}\sigma^{-1}).
\]
 In particular, if $\sigma=\zeta e$  for some  $\zeta>0,$ then this density becomes
\begin{equation}\label{ceigen}
(\dfrac{(2\pi)^{n-r}\Gamma(\frac{d}{2})^r}{(2\zeta)^{\frac{1}{2}\eta r}\Gamma_{\Omega}(\frac{1}{2}\eta)\Gamma_{\Omega}(\frac{rd}{2})})\exp\{-\frac{1}{2\zeta}\sum_{j=1}^r\xi_j\}\prod_{j=1}^r\xi_j^{\frac{1}{2}\eta-\frac{n}{r}}\prod_{j<i}(\xi_j-\xi_i)^d.
\end{equation}
$b)$~ If $\sigma=\zeta e$, then the density of the maximal invariant $\mathrm {Eig}(\x)$ is given by
\[
c_0\dfrac{\prod_{j<i}(\xi_j-\xi_i)^d\prod^r_{j=1}{\xi_j}^{\frac{1}{2}\eta-\frac{n}{r}}}{(2\zeta)^{\frac{1}{2}\eta r}\Gamma_{\Omega}(\frac{\eta}{2})}\exp\{-\frac{1}{2\zeta}\sum^r_{j=1}\xi_j\}\exp\{-\frac{1}{2}tr(\epsilon)\}{_0F_1}(\frac{1}{2}\eta,x, \frac{1}{4\zeta} \epsilon).
\]
\end{Cor}
\begin{Rem}
 Equation \eqref{ceigen} is essentially proved in \cite{Massam1997}.
\end{Rem}
\begin{Ex} According to the classification of irreducible symmetric cones we have:\\
$i)$~If $\Omega=\PD{N}$, the cone of $N\times N$ positive definite matrices over $\R$, then $n=N(N+1)/2$,  $r=N$ and $d=1$ Thus the density of the joint distribution of the ordered eigenvalues of the  real Wishart distribution $\wish_N (\eta,\zeta I_N)$, with shape parameter $\eta>N-1$ and scale parameter $\zeta I_N$,  is given by
\begin{equation*}
\dfrac{2^{\frac{N(N-1)}{2}}\pi^{\frac{N^2}{2}}}{(2\zeta)^{\frac{1}{2}\eta N}\Gamma_{\Omega}(\frac{1}{2}\eta)\Gamma_{\Omega}(\frac{N}{2})}\;\exp\{-\frac{1}{2\zeta}\sum_{j=1}^N\xi_j\}\prod_{j=1}^N\xi_j^{\frac{1}{2}\eta-\frac{N+1}{2}}\prod_{j<i}(\xi_j-\xi_i).
\end{equation*}
or
\[
\dfrac{\pi^{\frac{N^2}{2}}\zeta^{-N\eta}}{\Gamma_{N}(\frac{1}{2}\eta)\Gamma_{N}(\frac{N}{2})}\;\exp\{-\frac{1}{\zeta}\sum_{j=1}^{N}\xi_j\}\prod_{j=1}^N(\xi_j)^{\eta-\frac{N+1}{2}}\prod_{j<i}(\xi_j-\xi_i),
\]
where $\Gamma_N(\eta)=\pi^{\frac{N(N-1)}{4}}\prod_{j=1}^N\Gamma(\eta-\frac{1}{2}(j-1))$ is the multivariate gamma function \cite{Muirhead1982}.\\
$ii)$~If  $\Omega=\mathrm{PD}_N(\Complex)$, the cone of positive definite matrices over $\Complex$, then $ n=N^2$,  $r=N$  and  $d=2.$
Therefore, the density of the joint distribution of the ordered eigenvalues of the complex Wishart distribution $\Complex\wish_N(\eta,\zeta I_N),$  with
shape parameter $\eta>2N-2$ and scale parameter $\zeta I_N,$
is given by
\begin{equation*}
\dfrac{(2\pi)^{N^2-N}}{(2\zeta)^{\frac{1}{2}\eta N}\Gamma_{\Omega}(\frac{1}{2}\eta)\Gamma_{\Omega}(N)}\; \exp\{-\frac{1}{2\zeta}\sum_{j=1}^N\xi_j\}\prod_{j=1}^N\xi_j^{\frac{1}{2}\eta-N}\prod_{j<i}(\xi_j-\xi_i)^2.
\end{equation*}
$iii)$~If  $\Omega=\mathrm{PD}_N(\mathbb{H})$, the cone of positive definite matrices over $\mathbb{H}$,   then  $n=N(2N-1)$, $r=N$ and $d=4.$ Therefore,  the density of the joint distribution of the ordered
eigenvalues of the  quaternion Wishart distribution  $\mathbb{H}\wish_N(\eta,\zeta I_N)$, with shape parameter $\eta>4N-4$ and scale parameter  $\zeta I_N$  is given by
\begin{equation*}
\dfrac{(2\pi)^{2N^2-2N}}{(2\zeta)^{\frac{1}{2}\eta N}\Gamma_{\Omega}(\frac{1}{2}\eta)\Gamma_{\Omega}(2N)} \; \exp\{-\frac{1}{2\zeta}\sum_{j=1}^N\xi_j\}\prod_{j=1}^N\xi_j^{\frac{1}{2}\eta-2N+1}\prod_{j<i}(\xi_j-\xi_i)^4.
\end{equation*}
$iv)$~If  $\Omega$  is the Lorentz cone $\mathbb{L}_n$,  then the dimension of the Jordan algebra is  ~$n$,  $r=2$ and  $d=n-2$.  Therefore, the density of the joint distribution of the ordered eigenvalues of the Lorentz Wishart distribution  $\mathbb{L}\wish_n(\eta,\zeta e),$ with shape parameter $\eta>n-2$ and scale parameter $\zeta e,$ is given by
\begin{equation*}
\dfrac{\Gamma(\frac{n-2}{2})}{(2\zeta)^{\eta }\Gamma(\frac{1}{2}\eta)\Gamma(\frac{1}{2}\eta-\frac{n-2}{2})(n-3)!} \; \exp\{-\frac{1}{2\zeta}(\xi_1+\xi_2)\}(\xi_1\xi_2)^{\frac{1}{2}\eta-\frac{1}{2}n}(\xi_2-\xi_1)^{n-2}.
\end{equation*}
$v)$~If  $\Omega=\mathbf{PD}_3(\mathbb{O})$, then $n=27$, $r=3$ and $d=8.$ Thus the density of the joint distribution of the ordered eigenvalues of the octonion  Wishart distribution $\mathbb{O}W(\eta,\zeta e)$, with shape parameter $\eta>16$ and scale parameter $\zeta e,$ is given by
\begin{equation*}
\dfrac{36(2\zeta)^{-\frac{3}{2}\eta}}{\Gamma(\frac{1}{2}\eta)\Gamma(\frac{1}{2}\eta-4)\Gamma(\frac{1}{2}\eta-8)11!7!} \;\exp\{-\frac{1}{2\zeta}\sum_{j=1}^3\xi_j\}(\xi_1\xi_2\xi_3)^{\frac{1}{2}\eta-9}\prod_{1> i<j> 3}(\xi_j-\xi_i)^3.
\end{equation*}
\end{Ex}
We can also use Proposition \ref{dlatent} to  compute the  joint density of the eigenvalues of the beta distribution.
 \begin{Pro}\label{dlatent}
 Let $\x\sim \wish_{\Omega}(\eta_1, e,\epsilon_1)$, \  $\y\sim \wish_{\Omega}(\eta_2, e, \epsilon_2)$ \ and \   $\bo{z}=\y^{-1}\star \x$.\\
$a)$~ If $\epsilon_1=0$\ and \  $\epsilon_2=\epsilon$, then the probability density of   $\mathrm{Eig}(\bo{z})$\  is given by
 \begin{eqnarray}
\notag&&c_0\frac{1}{B_{\Omega}(\frac{1}{2}\eta_1,\frac{1}{2}\eta_2)}\prod_{j<i}(\zeta_j-\zeta_i)^d\prod_{j=1}^r\zeta_j^{\frac{1}{2}\eta_1-\frac{n}{r}}\prod_{j=1}^r(1+\zeta_j)^{-\frac{\eta_1+\eta_2}{2}}\exp\{-\frac{1}{2}\tr(\epsilon)  \}\\
&&\cdot{_{1}F_1}(\frac{\eta_1+\eta_2}{2}, \frac{1}{2}\eta_2,(e+z)^{-1},\frac{1}{2}\epsilon),
\end{eqnarray}
where $\zeta_1(z)>\cdots> \zeta_r(z)$ are the eigenvalues of $z$.\\
$b)$~ If $\epsilon_1=\epsilon$\  and \ $\epsilon_2=0$, the density of the maximal invariant \ $\mathrm{Eig}(\bo{z})$\ is given by
\begin{eqnarray}
\notag&&c_0\frac{1}{B_{\Omega}(\frac{1}{2}\eta_1,\frac{1}{2}\eta_2)}\prod_{j<i}(\zeta_j-\zeta_i)^d\prod_{j=1}^r\zeta_j^{\frac{1}{2}\eta_1-\frac{n}{r}}\prod_{j=1}^r(1+\zeta_j)^{-\frac{\eta_1+\eta_2}{2}}\exp\{-\frac{1}{2}\tr(\epsilon)  \}\\
&&\cdot{_0F_1}(\frac{1}{2}\eta_1, \  z, \ \frac{1}{4}\epsilon).
\end{eqnarray}
\end{Pro}
\begin{proof}
$a)$  Applying Theorem \ref{deigen} to the density of the random vector $\bo{z}$ given by Proposition \ref{p1}, the density of  $\mathrm{Eig}(\bo{z})$ is then
\begin{eqnarray*}
&&c_0\frac{1}{B_{\Omega}(\frac{1}{2}\eta_1,\frac{1}{2}\eta_2)}\prod_{j<i}(\zeta_j-\zeta_i)^d\prod_{j=1}^r\zeta_j^{\frac{1}{2}\eta_1-\frac{n}{r}}\prod_{j=1}^r(1+\zeta_j)^{-\frac{\eta_1+\eta_2}{2}}\exp\{-\frac{1}{2}\tr(\epsilon)  \}\\
&&\cdot{_{1}F_1}(\frac{\eta_1+\eta_2}{2}, \frac{1}{2}\eta_2,(e+z)^{-1},\frac{1}{2}\epsilon).
\end{eqnarray*}
The proof of part $b)$ is similar.
 \end{proof}
 \begin{Cor} \label{Releigen}
  Suppose  $\x\sim \wish_{\Omega}(\eta_1, \sigma, \epsilon_1)$ and $\y\sim \wish_{\Omega}(\eta_2, \sigma, \epsilon_2)$.\\
$a)$~If $\epsilon_1=0$\ and \ $\epsilon_2=\epsilon$, then   the joint density of the distribution of the ordered eigenvalues  $\zeta_1> \cdots> \zeta_r$,  of $\x$  with respect to  $\y$,  is given by
\begin{eqnarray*}
&&c_0\frac{1}{B_{\Omega}(\frac{1}{2}\eta_1,\frac{1}{2}\eta_2)}\prod_{j<i}(\zeta_j-\zeta_i)^d\prod_{j=1}^r\zeta_j^{\frac{1}{2}\eta_1-\frac{n}{r}}\prod_{j=1}^r(1+\zeta_j)^{-\frac{\eta_1+\eta_2}{2}}\exp\{-\frac{1}{2}\tr(\epsilon)  \}\\
&&\cdot{_{1}F_1}(\frac{\eta_1+\eta_2}{2}, \frac{1}{2}\eta_2,(e+z)^{-1},\frac{1}{2}\sigma\star\epsilon).
\end{eqnarray*}
$b)$~If $\epsilon_1=\epsilon$\ and $\epsilon_2=0$, then the joint density of the distribution of the ordered eigenvalues  $\zeta_1> \cdots> \zeta_r$,  of $\x$  with respect to  $\y$,  is given by
\begin{eqnarray*}
&&c_0\frac{1}{B_{\Omega}(\frac{1}{2}\eta_1,\frac{1}{2}\eta_2)}\prod_{j<i}(\zeta_j-\zeta_i)^d\prod_{j=1}^r\zeta_j^{\frac{1}{2}\eta_1-\frac{n}{r}}\prod_{j=1}^r(1+\zeta_j)^{-\frac{\eta_1+\eta_2}{2}}\exp\{-\frac{1}{2}\tr(\epsilon)  \}\\
&&\cdot{_0F_1}(\frac{1}{2}\eta_1, \  z, \ \frac{1}{4}\epsilon).
\end{eqnarray*}
\end{Cor}
\begin{proof} The eigenvalues of $\x$  with respect to  $\y$ are the roots of the polynomial
$p(\zeta)=\det(\x-\zeta\y)$, which are identical with the roots of the polynomial $
p_1(\zeta)=\det(\sigma^{-1}\star \x-\zeta\sigma^{-1}\star\y)$. 
Therefore, we may assume  that $\x\sim \wish_{\Omega}(\eta_1, e)$ and $\y\sim \wish_{\Omega}(\eta_2, e, \sigma\star\epsilon)$. Now the eigenvalues of $\x$  with respect to  $\y$ are the same as the eigenvalues of $\y^{-1}\star\x$ and the result follows from Proposition \ref{dlatent}.
 \end{proof}

The following lemma  applies directly to the testing problem that we will discuss in the next section. 
 \begin{Lem}\label{Lemma}
 Suppose $\x\sim \wish_{\Omega}(\eta_1,\sigma)$\  and\  $\y\sim\wish_{\Omega}(\eta_2,\sigma,\epsilon)$\   are  two independently distributed random vectors with values in $\Omega$. Then the  joint density of the distribution of the ordered eigenvalues,   $ l_1>\cdots > l_r$, of the random vector  $\bo{r}=(\x-\y)/2$  with respect to  the random vector  $\bo{s}=(\x+\y)/2$  is given by
  \begin{align*}
  &c_0\frac{2^{n-\frac{\eta_1+\eta_2}{2}}}{B_{\Omega}(\frac{1}{2}\eta_1,\frac{1}{2}\eta_2)}\prod_{j<i}(l_j-l_i)^d\prod_{j=1}^r(1-l_j)^{\frac{1}{2}\eta_1-\frac{n}{r}}\prod_{j=1}^r(1+l_j)^{\frac{1}{2}\eta_2-\frac{n}{r}}\\
&\cdot\exp\{-\frac{1}{2}\tr(\epsilon)\}{_{1}F_1}(\frac{\eta_1+\eta_2}{2}, \frac{1}{2}\eta_2,\frac{1}{2}(s-r)\star s^{-1}, \ \frac{1}{2}\sigma\star\epsilon).
\end{align*}\label{latent}
 \end{Lem}
 \begin{proof}
   Let $\zeta_1>\cdots> \zeta_r>0$ be the  ordered eigenvalues of $\x$  with respect to  $\y$. Hence there is a $g\in G$ and a Jordan frame $c_1,\cdots,c_r$ in $J$ such that $g\y=e$ and $g\x=\sum_{j=1}^r\zeta_jc_j$, where \ $\zeta_1>\cdots>\zeta_r$ are the eigenvalues of~$\x$  with respect to  $\y$. Therefore, $g\bo{s}=\sum_{j=1}^r1/2(1+\zeta_j)c_j$ and  $g\bo{r}=\sum_{j=1}^r1/2(1-\zeta_j)c_j.$ Without loss of generality, we may replace $\bo{r}$ with  $g\bo{r}$ and $\bo{s}$  with  $g\bo{s}$. Thus we have
\[
\bo{s}^{-1}\star \bo{r}=\sum_{j=1}^r(\dfrac{1-\zeta_j}{1+\zeta_j})c_j.
\]
  Therefore,  the ordered eigenvalues of $\bo{r}$  with respect to  $\bo{s}$ are given by
\[
-1<l_1=\frac{1-\zeta_r}{1+\zeta_r}<\cdots<\frac{1-\zeta_1}{1+\zeta_1}=l_r<1
\]
and, consequently,  under the transformation
\begin{eqnarray*}
\mathbb{R}_+^r&\rightarrow&\{(l_1,\cdots,l_r) :\ -1<l_r<\cdots< l_1<1\}\\
(\zeta_1,\cdots,\zeta_r)&\mapsto&(\frac{1-\zeta_r}{1+\zeta_r},\ldots,\frac{1-\zeta_1}{1+\zeta_1}),
\end{eqnarray*}
we can compute the joint density of  $l_1,\ldots, l_r$.  The inverse of this map is
\begin{eqnarray*}
\{(l_1,\cdots,l_r):~-1<l_r<\cdots<l_1<1\}&\rightarrow&\R_+^r\\
(l_1,\cdots,l_r)&\mapsto&(\frac{1-l_r}{1+l_r},\cdots,\frac{1-l_1}{1+l_1}),
\end{eqnarray*}
with  Jacobian $2^r\prod_{j=1}^r(1+l_j)^{-2}$. With this  and from Corollary \ref{Releigen}, we have the  joint density of $l_1> \cdots > l_r$ is
\begin{align*}
&c_0\frac{2^r\prod_{j=1}^r(1+l_j)^{-2}}{B_{\Omega}(\frac{1}{2}\eta_1,\frac{1}{2}\eta_2)}\prod_{j<i}(\frac{2(l_j-l_i)}{(1+l_j)(1+l_i)})^d
\prod_{j=1}^r(\frac{1-l_j}{1+l_j})^{\frac{1}{2}\eta_1-\frac{n}{r}}
\prod_{j=1}^r(1+\frac{1-l_j}{1+l_j})^{-\frac{\eta_1+\eta_2}{2}}\\
&\cdot\exp\{-\frac{1}{2}\tr(\epsilon)  \}{_{1}F_1}(\frac{\eta_1+\eta_2}{2}, \frac{1}{2}\eta_2,(e+(s-r)^{-1}\star(s+r))^{-1},\ \frac{1}{2}\sigma\star\epsilon)\\
&=\frac{c_{0}}{B_{\Omega}(\frac{1}{2}\eta_1,\frac{1}{2}\eta_2)}\frac{2^{r(r-1)\frac{d}{2}}}{(1+l_j)^{(r-1)d}}\prod_{j<i}(l_j-l_i)^d\prod_{j=1}^r(\frac{1-l_j}{1+l_j})^{\frac{1}{2}\eta_1-\frac{n}{r}}\prod_{j=1}^r(\frac{2}{1+l_j})^{-\frac{\eta_1+\eta_2}{2}}\\
&\cdot\frac{2^r}{\prod_{j=1}^r(1+l_j)^2}\exp\{-\frac{1}{2}\tr(\epsilon)  \}{_{1}F_1}\frac{(\eta_1+\eta_2}{2}, \frac{1}{2}\eta_2,\frac{1}{2}(s-r)\star s^{-1}, \ \frac{1}{2}\sigma\star\epsilon)\\
&=c_0\frac{2^{n-\frac{\eta_1+\eta_2}{2}}}{B_{\Omega}(\frac{1}{2}\eta_1,\frac{1}{2}\eta_2)}\prod_{j<i}(l_j-l_i)^d\prod_{j=1}^r(1-l_j)^{\frac{1}{2}\eta_1-\frac{n}{r}}\prod_{j=1}^r(1+l_j)^{\frac{1}{2}\eta_2-\frac{n}{r}}\\
&\cdot\exp\{-\frac{1}{2}\tr(\epsilon)\}{_{1}F_1}(\frac{\eta_1+\eta_2}{2}, \frac{1}{2}\eta_2,\frac{1}{2}(s-r)\star s^{-1}, \ \frac{1}{2}\sigma\star\epsilon),
\end{align*}
where we have used the fact that $\bo{z}=(\bo{s}-\bo{r})^{-1}\star(\bo{s}+\bo{r})$. 
\end{proof}
\section{A generalization of Bartlett's test}\label{sec5}
In classical multivariate analysis Bartlett's test is used for testing  equality of variances of two, or more,  normally distributed  populations based on their observed samples. Since the sample covariance matrix of a normally distributed population follows a Wishart distribution, Bartlett's test is indeed  a testing  problem concerning the scale parameters of two, or more, Wishart  models. In this section we use the results obtained in previous  sections to show how we can set up a similar testing problem for  scale parameter of a Wishart model parameterized over a symmetric cone.\\

 Suppose $\eta > (r-1)d$ \  is a known parameter, and  consider the statistical model
 \[
 (\wish_{\Omega}(\eta,\sigma_1)\otimes \wish_{\Omega}(\eta,\sigma_2)\in \mathbb{P}(\Omega\times \Omega):~ (\sigma_1,\sigma_2)\in \Omega\times \Omega)
  \]
  and its submodel 
  \[ (\wish_{\Omega}(\eta,\sigma)\otimes \wish_{\Omega}(\eta, \sigma)\in \mathbb{P}(\Omega\times\Omega):~\sigma\in \Omega).
  \]
  We wish to test the null hypothesis  $H_0 : \sigma_1=\sigma_2=\sigma \quad \text{vs.}\quad H: \sigma_1\neq\sigma_2.$
  \begin{The}\label{btest}
  For the observation \ $(x,y) \in \Omega\times \Omega$\  the maximum likelihood estimator of  $\sigma$ under $H_0$ is  $s(x,y)=(x+y)/2,$  and the likelihood ratio statistic for testing  $H_0$~vs.~$H$  is  $\prod_{j=1}^r(1-{l_j}^2)^{\frac{1}{2}\eta}$, where  $l_1>l_2>\cdots>l_r$ are the eigenvalues of $\bo{r}$  with respect to  $\bo{s}.$ Furthermore, under the null hypothesis $H_0,$ the statistics $s(\x,\y)$ and $\pi(\x,\y)=(l_1,\cdots,l_r)$ are independently distributed, $s(\x,\y)\sim \wish_{\Omega}(\eta,\sigma )$ and the density of $\pi(x,y)$  is given by
  \begin{equation*}
  c_0\frac{2^{n-\eta}}{B_{\Omega}(\frac{1}{2}\eta,\frac{1}{2}\eta)}\prod_{j<i}(l_j-l_i)^d\prod_{j=1}^r(1-l^2_j)^{\frac{1}{2}\eta-\frac{n}{r}}.
\end{equation*}
  \end{The}
  \begin{proof}
Consider the diagonal action of $G$ on   $\Omega\times\Omega$  defined by  $(g,(x,y))\mapsto (gx,gy)$. Evidently, $\pi$ is surjective, and invariant under $G$. Even more, we show that $\pi$ is maximal invariant under $G$.  Suppose $\pi(x,y)=\pi(x',y')$. We may assume, without loss of generality, that $x=x'=e,$ and therefore, the eigenvalues of $y$ and $y'$ are
identical. Let $\sum_{j=1}^rl_jc_j$ and $\sum_{j=1}^rl_jc'_j$ be, respectively,  the spectral decompositions of $y$ and $y'$. Since $K$ acts transitively
on the set of Jordan frames of  $\Omega$,  we have $y$ and $y'$ are in the same orbit of $G$. It is known that under the null hypothesis  $H_0$  the statistic
$s(x,y)=(x+y)/2$  is the maximum likelihood estimator (MLE) of  $\sigma$, which is also invariant under the action of $G$.
Therefore, the independency of $s(\x,\y)$ and $\pi(\x,\y)$ follows from  \cite[Lemma3]{Andersson1983}. The density of $\pi(\x,\y)$ is  given by Lemma  \ref{Lemma} for $\eta_1=\eta_2=\eta$ and  $\epsilon=0$. To complete the proof let calculate the likelihood ratio test LR.  Since $(x,y)$  is the MLE of  $(\sigma_1,\sigma_2)$ we have
\begin{eqnarray*}
q(x,y)&=&\dfrac{\mathbb{L}_0(\widehat{\sigma}
|(x,y))}{\mathbb{L}(({\widehat{\sigma_1},\widehat{\sigma_2})}|(x,y))}
=\dfrac{\det(\frac{x+y}{2})^{-\eta}}
{\det(x)^{-\frac{1}{2}\eta}\det(y)^{-\frac{1}{2}\eta}}
=2^{\eta r}\dfrac{\det(x)^{\frac{1}{2}\eta}\det(y)^{\frac{1}{2}\eta}}{\det(x+y)^{\eta}}\\
&=&2^{\eta r}\dfrac{\det(x^{-1}\star y)^{\frac{1}{2}\eta}}{\det(e+x^{-1}\star y)^{\eta}}
=\left(4^r\prod_{j=1}^r\dfrac{\zeta_j}{(1+\zeta_j)^2}\right)^{\frac{1}{2}\eta}
=\left(4^r\prod_{j=1}^r\frac{(\dfrac{1-l_j}{1+l_j})}{(1+\frac{1-l_j}{1+l_j})^2}\right)^{\frac{1}{2}\eta}\\
&=&\prod_{j=1}^r(1-{l_j}^2)^{\frac{1}{2}\eta}.
\end{eqnarray*}
  \end{proof}
 \section*{Acknowledgments} I would like to thank my thesis advisor Prof. Steen Andersson at Indiana University for guiding me through this work.

\bibliographystyle{plain}

\end{document}